\documentclass[12pt,reqno]{amsart} 
\usepackage{amsmath}
\usepackage{amssymb}
\newtheorem{thm}{Theorem}

\newtheorem{prop}[thm]{Proposition}

\newcommand{\vertiii}[1]{{\left\vert\kern-0.25ex\left\vert\kern-0.25ex\left\vert
#1 \right\vert\kern-0.25ex\right\vert\kern-0.25ex\right\vert}}

\def \lim   {\text {\rm lim}}

\def \tr    {\text {\rm tr}}

\def \argmin {\text {\rm argmin}}

\begin{document}

\title[Inequalities for the Wasserstein mean]{Inequalities for the Wasserstein mean of positive definite matrices}

\author[Rajendra Bhatia]{Rajendra Bhatia}
\address{Ashoka University, Sonepat\\ Haryana, 131029, India}

\email{rajendra.bhatia@ashoka.edu.in}

\author[Tanvi Jain]{Tanvi Jain}
\address{Indian Statistical Institute\\ New Delhi 110016, India}
\email{tanvi@isid.ac.in}

\author[Yongdo Lim]{Yongdo Lim}

\address{Department of Mathematics, Sungkyunkwan University\\ Suwon 440-746, Korea}

\email{ylim@skku.edu}




\maketitle

 \centerline{\it Dedicated to Fumio Hiai, and to the
memory of Denes Petz}

\begin{abstract}
We prove majorization inequalities for different means of positive definite matrices.
These include the Cartan mean (the Karcher mean),
the log Euclidean mean, the Wasserstein mean and the power mean.
\end{abstract}

\section{Introduction}
Let  $\mathbb{P}$ be the space of $n\times n$ complex positive
definite matrices. The \emph{Bures-Wasserstein distance} on ${\Bbb
P}$ is the metric defined as
\begin{equation}
d(A,B) = \left [\tr\,(A+B) - 2\, \tr\,\left(A^{1/2}BA^{1/2}
\right)^{1/2}\right ]^{1/2}.\label{eq1}
\end{equation}

Let $A_{1},\dots,A_{m}$ be given positive definite matrices and let
$ w =(w_{1},\dots,w_{m})$ be a vector of weights; i.e., $w_j\geq 0$
and $\sum_{j=1}^{m}w_{j}=1.$ Then the (\emph{weighted})
\emph{Wasserstein mean, or the Wasserstein barycentre} of
$A_{1},\dots,A_{m}$ is defined as
\begin{equation}
\Omega (w; A_1, \ldots, A_m) = \underset{X\in {\Bbb
P}}{\argmin}\,\,\, \sum\limits^{m}_{j=1} w_j d^2(X,A_j). \label{eq2}
\end{equation}
It can be shown that the function on ${\Bbb P}$ defined by the sum
on the right hand side of (\ref{eq2}) has a unique minimizer, and
the notation $\argmin$ is used for this minimizer. It is also known
that $\Omega$ is the unique positive definite solution of the
equation
\begin{equation}
X = \sum_{j=1}^{m}w_{j}\left(X^{1/2}A_{j}X^{1/2}\right)^{1/2}.
\label{eq3}
\end{equation}

In the special case $m=2,$ writing $(A_{1},A_{2})=(A,B)$ and
$(w_{1},w_{2})=(1-t,t),$ where $0\leq t\leq 1,$ we have an explicit
formula for $\Omega.$ Denoting this by $A\, \lozenge_t  B$ we have
\begin{equation}
A\, \lozenge_t
B=(1-t)^2A+t^2B+t(1-t)\left[(AB)^{1/2}+(BA)^{1/2}\right].
\label{eq4}
\end{equation}
The matrix $AB$ has only positive eigenvalues and, therefore, has a
unique square root with positive eigenvalues. This is the matrix
$(AB)^{1/2}$ in the expression above. The metric $d$ in (\ref{eq1})
is the distance corresponding to an underlying Riemannian metric on
${\Bbb P}$, and (\ref{eq4}) is an equation for the geodesic segment
joining two points $A$ and $B$ in the manifold ${\Bbb P}.$ The
special choice $t=1/2$ gives the midpoint of this geodesic. This is
denoted by
\begin{equation}
A\, \lozenge B=\frac{1}{4}\left[A+B+(AB)^{1/2}+(BA)^{1/2}\right],
\label{eq5}
\end{equation}
and can  be thought of as the Wasserstein mean of $A$ and $B.$

We refer the reader to our recent article \cite{BJL} for the relevance
and importance of the Wasserstein metric, mean and barycentre in
various areas like quantum information, statistics, optimal
transport and Riemannian geometry.

More familiar in the matrix theory literature, and much studied in the
past few years, has been the geometric mean, variously called the
Cartan mean, the Karcher mean or the Riemannian mean. To describe
this we start with the \emph{Cartan metric}
\begin{equation}
\delta(A,B)=||\log A^{-1/2}BA^{-1/2}||_{2}, \label{eq6}
\end{equation}
where $||X||_2=({\mathrm{tr}}X^*X)^{1/2}$ is the Frobenius norm on
matrices. The \emph{weighted Cartan mean} (or the \emph{weighted
geometric mean}) of $A_{1},\dots,A_{m}$ is defined as
\begin{equation}
G( w ;A_{1},\dots,A_{m})=\underset{X\in {\Bbb P}}{\argmin}\,\,\,
\sum\limits^{m}_{j=1}w_j \delta^2(X,A_{j}). \label{eq7}
\end{equation}
The (unique) solution of this minimization problem is also the
positive definite solution of the equation
\begin{equation}
\sum_{j=1}^mw_{j}\log(X^{-1/2}A_{j}X^{-1/2})=0. \label{eq8}
\end{equation}
The mean $G$ was introduced in matrix analysis by M. Moakher
\cite{m1} and R. Bhatia and J. Holbrook \cite{bh} as the solution of
the long standing problem of defining an appropriate geometric mean
of several positive definite matrices, and has since then been an
object of intense study.

Analogous to (\ref{eq4}) the equation of the geodesic segment
joining $A$ and $B$ with respect to the metric $\delta$ is
\begin{equation}
A\#_tB=A^{1/2}(A^{-1/2}BA^{-1/2})^{t}A^{1/2}, \ \ \ \ \ 0\leq t\leq
1. \label{eq9}
\end{equation}
This is also called the $t$-weighted geometric mean of $A$ and $B$.
When $t=1/2,$ this reduces to
\begin{equation}
A\#B=A^{1/2}(A^{-1/2}BA^{-1/2})^{1/2}A^{1/2}, \label{eq10}
\end{equation}
and is called the geometric mean of $A$ and $B.$

Presented with the two important means $\Omega$ and $G$, arising
from two different geometries on ${\Bbb P}$, it is natural to ask
for comparisons between them. The two metrics $d$ and $\delta$ are
strikingly different. Endowed with the metric $d$, the manifold
${\Bbb P}$ has nonnegative curvature \cite{t,BJL}; with the metric
$\delta$ it has nonpositive curvature \cite{rbh1}. The Cartan mean
has nice order properties: it is monotonic in variables
$A_{1},\dots,A_{m}$ \cite{ll} and lies between the harmonic and
arithmetic means; i.e.,
\begin{equation}
\left(\sum_{j=1}^mw_{j}A_{j}^{-1}\right)^{-1}\leq
G( w ;A_{1},\dots,A_{m})\leq \sum_{j=1}^mw_{j}A_{j}, \label{eq11}
\end{equation}
where $\leq$ is the L\"owner ordering; $A\leq B$ if $B-A$ is
positive semidefinite. It has been shown in \cite{BJL} that the
Wasserstein mean $\Omega$ is not monotonic in the variables
$A_{1},\dots,A_{m}$. While it is always bounded above by the
arithmetic mean, it may not be bounded below by the harmonic mean. \emph{A fortiori} the operator inequality $G\leq \Omega$ does not
hold. Somewhat surprisingly, good comparison theorems in terms of
log majorizations can be proved, and that is the main purpose of
this paper.

Let $x = (x_1, \ldots, x_n)$ and $y= (y_1, \ldots, y_n)$ be two
$n$-tuples of nonnegative numbers. Let $x_1^{\downarrow} \geq
x_2^{\downarrow} \geq \ldots \geq x_n^{\downarrow}$  be the
decreasing rearrangement of $x_{1},\dots,x_{n}.$  If for all $1\leq
k\leq n$
\begin{equation}
\prod^{k}_{j=1} x_j^{\downarrow} \,\,\, \leq
\prod^{k}_{j=1}y_j^{\downarrow},\,\,\, \label{eq12}
\end{equation}
we say that $x$ is  weakly log majorised by $y,$ and write this as
 \begin{equation}x \underset{w\log}{\prec} y. \label{eq13}
 \end{equation}
 If, in addition to (\ref{eq12}) we also have
\begin{equation}
\prod^{n}_{j=1} x_j^{\downarrow} \,\,\, =
\prod^{n}_{j=1}y_j^{\downarrow},\,\,\, \label{eq14}
\end{equation}
 we say $x$ is  log
majorised by $y,$ and write this as
\begin{equation}
x\underset{\log}{\prec} y.\label{eq15}
\end{equation}

Let $\lambda (A)$ stand for  the eigenvalue $n$-tuple of $A$. We
will show that $\lambda(G)\underset{\log}{\prec} \lambda(\Omega).$
In fact, we will prove a stronger result that involves another mean
called the \emph{log Euclidean mean}. This is the matrix defined as
\begin{equation}
L( w ;A_{1},\dots,A_{m})=\exp\left(\sum_{j=1}^{m}w_{j}\log
A_{j}\right).\label{eq16}
\end{equation}
If we replace the metric $\delta$ in (\ref{eq6}) by
\begin{equation}
\delta_L(A,B)=||\log A-\log B||_{2},\label{eq17}
\end{equation} and consider the corresponding minimization  problem
(\ref{eq7}), then the solution is the log Euclidean mean $L.$ In the
special case when $A_{1},\dots,A_{m}$ are pairwise commuting
positive definite matrices, we have
\begin{equation}
G=L=\prod_{j=1}^mA_{j}^{w_{j}}.\label{eq18}
\end{equation}

The first half of the following theorem was proved by Hiai and Petz
\cite{hp}, the second is new.

\begin{thm}\label{thm1}
Let $G=G(w;A_1,\ldots,A_m),L=L( w ;A_{1},\dots,A_{m})$ and
$\Omega=\Omega( w ;A_{1},\dots,A_{m})$  be the Cartan, the log
Euclidean, and the  Wasserstein means of positive definite matrices
$A_{1},\dots,A_{m}.$  Then
\begin{equation}
\lambda(G)\underset{\log}{\prec} \lambda(L)\underset{w\log}{\prec}
\lambda(\Omega).\label{eq19}
\end{equation}
\end{thm}

There is another family of means that is important in this context.
The \emph{$p$-th power mean}  of $A_{1},\dots,A_{m}$ is defined as
\begin{equation}
Q_p( w ;A_{1},\dots,A_{m})=\left(\sum_{j=1}^{m}w_{j}A_{j}^p\right)^{1/p},
\ \ \ -\infty<p<\infty.\label{eq20}
\end{equation}
For $p=0,$ this quantity  is to be interpreted as a limit. It was
shown by Bhagwat and Subramanian \cite{bs} that
\begin{equation}
{\underset{p\to
0^{\pm}}{\lim}}Q_p(w;A_1,\ldots,A_m)=L(w;A_1,\ldots,A_m).
\label{eq21}
\end{equation}

Of special interest to us here is the case $p=1/2.$ When
$A_{1},\dots,A_{m}$ commute we have

\begin{equation}
\Omega( w ;A_{1},\dots,A_{m})=\left(\sum_{j=1}^mw_jA_{j}^{1/2}\right)^2=Q_{1/2}(w;A_1,\ldots,A_m).
\label{eq22}
\end{equation}
Thus $\Omega$ and $Q_{1/2}$ can be regarded as two different
noncommutative extensions of the same object, and for this reason it
is natural to ask for comparisons between them.

We remark here that just as we have realized the means $    \Omega,G$ and
$L$ as solutions of least square problems for certain metrics on
${\Bbb P},$ so can be done for $Q_{1/2}.$  The Hellinger distance or
the Bhattacharya distance  between probability vectors
 $p = (p_1,\ldots, p_n)$ and $q = (q_1,\ldots, q_n)$ is defined as
\begin{equation*}
d(p,q) =||p^{1/2}-q^{1/2}||_{2}= \left[\sum_{i=1}^{n}
\left(\sqrt{p_i}-\sqrt{q_i}\right)^2 \right ]^{1/2}. \label{eq5a}
\end{equation*}
A straightforward extension to positive definite matrices is the
distance
\begin{equation}
d_H(A,B) =||A^{1/2}-B^{1/2}||_{2}=
\left[\tr\,(A+B)-2\tr\,A^{1/2}B^{1/2}\right ]^{1/2}. \label{eq23}
\end{equation}
Compare this with the Bures-Wasserstein distance (\ref{eq1}). It can
be seen that the solution to the least squares problem (\ref{eq2})
when $d$ is replaced by $d_H$ is the mean
$Q_{1/2}( w ;A_{1},\dots,A_{m}).$

For the comparison between the means $Q_{1/2}$ and $\Omega$ we have
the following result.

\begin{thm}\label{thm2}
For all positive definite matrices $A_{1},\dots,A_{m}$ and weights
$w_{1},\dots,w_{m}$, we have
\begin{equation}
||Q_{1/2}( w ;A_{1},\dots,A_{m})||_{p}\leq
||\Omega( w ;A_{1},\dots,A_{m})||_{p}, \label{eq24}
\end{equation}
for the Schatten $p$-norms with $p=1$ and $\infty.$ In the case
$m=2,$ the inequality $(\ref{eq24})$ holds also with $p=2.$
\end{thm}

Theorem \ref{thm2} establishes a part of the following:
\vspace{2mm}

 \noindent
 \textbf{Conjecture 1} The inequality (\ref{eq24}) is valid, more
 generally, for all unitarily invariant norms. In other words, we
 have the weak majorisation
 \begin{equation}
\lambda(Q_{1/2})\underset{w}{\prec} \lambda(\Omega).\label{eq25}
\end{equation}

We remark that when $m=2,$ and $(w_{1},w_{2})=(1/2, 1/2)$ the
conjecture says that
\begin{equation}
|||A+B+A^{1/2}B^{1/2}+B^{1/2}A^{1/2}|||\leq
|||A+B+(AB)^{1/2}+(BA)^{1/2}|||,\label{eq26}
\end{equation}
for every unitarily invariant norm, and Theorem \ref{thm2} includes
the statement that this is true for the $p$-norms with
$p=1,2,\infty.$ There is a parallel here with results proved in
\cite{bly}  which we discuss later in Section 3.

\section{Proofs}
The first majorisation in (\ref{eq19}) is mentioned as a remark at
the end of the paper \cite{hp}. Some of the ideas needed for the
proofs of the two majorisations are same. Therefore, for the reader's
convenience we include here proofs of both the majorisations in
(\ref{eq19}) and also of the propositions that go into them (with
some simplifications). The main ideas originate in the work of the
Japanese school beginning with T. Ando \cite{a}, followed by T. Ando
and Hiai \cite{ah}, and then by J. I. Fujii, M. Fujii, Y. Seo
\cite{ffs} and T. Yamazaki \cite{y}.

The function $f(A)=A^p$ on positive definite matrices is operator
convex if $1\leq p\leq 2,$ and operator concave if $0\leq p\leq 1.$
Operator convexity and concavity are characterized by Jensen type
inequalities called Hansen's inequalities \cite{han} and
\cite[Theorem 2.1]{HP}, which say that for every contraction $X$ we
have

 \begin{equation}
(X^*AX)^p\leq X^*A^pX, \qquad {\mathrm{if}}\ 1\leq p\leq 2,
\label{eq27}
 \end{equation}
and
 \begin{equation}
(X^*AX)^p\geq X^*A^pX, \qquad {\mathrm{if}}\ 0\leq p\leq 1.
\label{eq28}
 \end{equation}

 A standard technique in proving log majorisations like (\ref{eq19})
 is the use of antisymmetric tensor powers. This is so because if
 $\Lambda^kA$ denotes the $k$th antisymmetric tensor power of $A,$
 then
\begin{equation*}
\prod^{k}_{j=1}\lambda_j^{\downarrow}(A)=\lambda_1^{\downarrow}(\Lambda^kA),\qquad
1\leq k\leq n.
\end{equation*}
The map $A\mapsto \Lambda^kA$ is multiplicative, i.e.,
$\Lambda^k(AB)=(\Lambda^kA)(\Lambda^kB)$ and
$\Lambda^{k}A^p=(\Lambda^kA)^p, p\in (-\infty,\infty).$ So the
geometric mean $G$ is well disposed towards tensor powers. It is
clear from (\ref{eq9}) that
$\Lambda^k(A\#_tB)=(\Lambda^kA)\#_t(\Lambda^kB).$ The same property
holds for the several variable geometric mean
\begin{equation}
\Lambda^kG( w ;A_{1},\dots,A_{m})=G( w ;\Lambda^kA_1,\dots,\Lambda^kA_{m}).
\label{eq29}
\end{equation}
See \cite{bk} for a proof. Most of the other means do not behave as
well with respect to tensor powers. For the log Euclidean mean this
difficulty is circumvented through the following ingenious
proposition, proved by Ando and Hiai \cite{ah} for two variables and
extended by Fujii et al \cite{ffs} to several variables.

\begin{prop} We have
\begin{equation}
{\underset{p\to
0^+}{\lim}}G( w ;A_{1}^p,\dots,A_{m}^p)^{1/p}=L( w ;A_{1},\dots,A_{m}).\label{eq30}
\end{equation}
\end{prop}

\begin{proof}
By the arithmetic-geometric-harmonic mean inequalities (\ref{eq11}),
we have for every $p>0$
$$\left(\sum_{j=1}^{m}w_{j}A_{j}^{-p}\right)^{-1}\leq
G( w ;A_{1}^p,\dots,A_{m}^p)\leq \sum_{j=1}^{m}w_{j}A_{j}^p.$$
Take logarithms and use the fact that $\log$ is an operator monotone
function to get
\begin{eqnarray*}
-\log\left(\sum_{j=1}^mw_{j}A_{j}^{-p}\right)&\leq&\log
G( w ;A_{1}^p,\dots,A_{m}^p)\leq\log \left(
\sum_{j=1}^{m}w_{j}A_{j}^p\right).
\end{eqnarray*}
Multiplying by $1/p$ we get
\begin{eqnarray*}
\log\left(\sum_{j=1}^mw_{j}A_{j}^{-p}\right)^{-1/p}&\leq&\log
G( w ;A_{1}^p,\dots,A_{m}^p)^{1/p}\leq\log \left(
\sum_{j=1}^{m}w_{j}A_{j}^p\right)^{1/p}.
\end{eqnarray*}
Now take the limit as $p\to 0^+$ and use the result of Bhagwat and
Subramanian (\ref{eq21}) to obtain (\ref{eq30}).
\end{proof}

Using the equation (\ref{eq8}) that characterizes $G$, we see that
$G( w ;A_{1},\dots,A_{m})=I\Longleftrightarrow
\sum_{j=1}^mw_{j}\log A_{j}=0\Longleftrightarrow
\sum_{j=1}^mw_{j}\log A_{j}^p=0$ for all $p>0\Longleftrightarrow
G( w ;A_{1}^p,\dots,A_{m}^p)=I$ for all $p>0.$ The following
proposition handles the case $G\leq I.$ It was proved for two
variables by  Ando \cite{a} and for general $m$ by Yamazaki
\cite{y}.

\begin{prop} Suppose $G( w ;A_{1},\dots,A_{m})\leq I.$ Then
\begin{equation}
G( w ;A_{1}^p,\dots,A_{m}^p)\leq G( w ;A_{1},\dots,A_{m}),
\qquad 1\leq p<\infty, \label{eq31}
\end{equation} and
\begin{equation}
G( w ;A_{1}^p,\dots,A_{m}^p)\geq G( w ;A_{1},\dots,A_{m}),
\qquad 0\leq p\leq 1. \label{eq32}
\end{equation}
\end{prop}

\begin{proof} Let $A_{1},\dots,A_m\in {\Bbb P}$, and let $X=G( w ;A_{1},\dots,A_{m}).$
By the congruence invariance property of $G$
$$G( w ;X^{-1/2}A_{1}X^{-1/2},\dots,X^{-1/2}A_{m}X^{-1/2})=I.$$
So, by the remark preceding Proposition 4
$$G( w ;(X^{-1/2}A_{1}X^{-1/2})^p,\dots,(X^{-1/2}A_{m}X^{-1/2})^p)=I$$
for all $p>0.$

If $X\geq I,$ then $X^{-1/2}\leq I.$Then by Hansen's inequality
(\ref{eq27}), for $1\leq p\leq 2$ we have
$\left(X^{-1/2}A_{j}X^{-1/2}\right)^p\leq X^{-1/2}A_{j}^pX^{-1/2}.$
The monotonicity property of $G$ then gives
$$I\leq
G( w ;X^{-1/2}A_{1}^pX^{-1/2},\dots,X^{-1/2}A_{m}^pX^{-1/2}),$$
and the congruence invariance of $G$ shows that
$$X\leq G( w ;A_{1}^p,\dots,A_{m}^p).$$ We have shown that if
$G( w ;A_{1},\dots,A_{m})\geq I,$ then
$G( w ;A_{1},\dots,A_{m})\leq G( w ;A_{1}^p,\dots,A_{m}^p)$
for $1\leq p\leq 2.$

If $G( w ;A_{1},\dots,A_{m})\leq I,$ then
$G( w ;A_{1}^{-1},\dots,A_{m}^{-1})\geq I.$ So, by what we have
proved $G( w ;A_{1}^{-1},\dots,A_{m}^{-1})\leq
G( w ;A_{1}^{-p},\dots,A_{m}^{-p})$ for $1\leq p\leq 2.$
Inverting again we get the inequality (\ref{eq31}) for $1\leq p\leq
2.$ Using this the validity of the inequality can be established for
$2\leq p\leq 4,$ and successively for intervals beyond. The same
arguments can be used to prove (\ref{eq32}).
\end{proof}

We now turn to the proof of Theorem 1. All the means under
consideration here (arithmetic, geometric, harmonic, log Euclidean,
Wasserstein, and the $p$-th power mean) are homogeneous in variables
$A_{j};$ i.e., for all $\alpha>0$ we have
\begin{equation}
G( w ; \alpha A_{1},\dots, \alpha A_{m})=\alpha
G( w ;A_{1},\dots,A_{m}), \label{eq33}
\end{equation} etc. So, to prove an inequality like
$\lambda_1^{\downarrow}(G)\leq \lambda_1^{\downarrow}(L)$ it
suffices to show that $G\leq I$ whenever $L\leq I.$

Let $0<p<1$ and suppose
$$G( w ;A_{1}^p,\dots,A_{m}^p)^{1/p}\leq I.$$ Then
$$G( w ;A_{1}^p,\dots,A_{m}^p)\leq I.$$
It follows from Proposition 4 that
$$G( w ;A_{1},\dots,A_{m})=G( w ;(A_{1}^p)^{1/p},\dots,(A_{m}^p)^{1/p})\leq
G( w ;A_{1}^p,\dots,A_{m}^p)\leq I.$$ Hence, for all $0<p<1$ we
have
$$\lambda_1^{\downarrow}(G( w ;A_{1},\dots,A_{m}))\leq \lambda_1^{\downarrow}(G( w ;A_{1}^p,\dots,A_{m}^{p})^{1/p}).$$
Replacing $A_{j}$ by $\Lambda^k(A_{j})$ and using (\ref{eq29}) we
get
\begin{eqnarray*}
\prod_{j=1}^k\lambda_j^{\downarrow}(G( w ;A_{1},\dots,A_{m}))&=&\lambda_1^{\downarrow}(\Lambda^kG( w ;A_{1},\dots,A_{m}))\\
&=&\lambda_1^{\downarrow}(G( w ;\Lambda^kA_{1},\dots,\Lambda^kA_{m}))\\
&\leq&\lambda_1^{\downarrow}(G( w ;(\Lambda^kA_{1})^p,\dots,(\Lambda^kA_{m})^p)^{1/p})\\
&=&\lambda_1^{\downarrow}(G( w ;(\Lambda^kA_{1}^p),\dots,(\Lambda^kA_{m}^p))^{1/p})\\
&=&\lambda_1^{\downarrow}(\Lambda^kG( w ;A_{1}^p,\dots,A_{m}^p)^{1/p})\\
&=&\prod_{j=1}^k\lambda_j^{\downarrow}(G( w ;A_{1}^p,\dots,A_{m}^p)^{1/p})
\end{eqnarray*}
for all $1\leq k\leq n.$ Letting $p\to 0$ we get from Proposition 3
$$\prod_{j=1}^{k}\lambda_j^{\downarrow}(G)\leq
\prod_{j=1}^{k}\lambda_j^{\downarrow}(L), \qquad 1\leq k\leq n.$$
For $k=n$ there is equality here as
${\mathrm{det}}(G)={\mathrm{det}}(L).$ This proves the first
majorisation in (\ref{eq19}).

To prove the second we start with the equation
\begin{equation}
\sum_{j=1}^{m}w_j(A_{j}\#\Omega^{-1})=I, \label{eq34}
\end{equation} that the Wasserstein mean $\Omega$ satisfies. This
follows from (\ref{eq3}) and (\ref{eq10}). Using the fact that the
function $f(A)=A^p$ is operator concave for $0<p<1$ we obtain from
(\ref{eq34}) the inequality
\begin{equation*}
\sum_{j=1}^{m}w_j(A_{j}\#\Omega^{-1})^p\leq I, \qquad 0<p<1.
\end{equation*}
Using the arithmetic-geometric mean inequality (\ref{eq11}) we get
from this
\begin{equation*}G( w ;(A_{1}\#\Omega^{-1})^p,\dots,(A_{m}\#\Omega^{-1})^p)\leq
I,\qquad 0<p<1,
\end{equation*}
and then using \eqref{eq29} we get
\begin{equation}
G( w ; (\Lambda^kA_1\#\Lambda^k\Omega^{-1})^p,\ldots, (\Lambda^kA_m\#\Lambda^k\Omega^{-1})^p)\le I,\ 0<p<1.\label{eq35}
\end{equation}

Now suppose  $A_{1},\dots,A_m$ are such that $\Lambda^k\Omega\leq
I.$ Then $\Lambda^k\Omega^{-1}\geq I,$ and hence
$\Lambda^kA_{j}\#\Lambda^k\Omega^{-1}\geq \Lambda^kA_{j}\#
I=\Lambda^kA_{j}^{1/2}.$ Hence for all $0<p<1$ we have from
the L\"owner-Heinz inequality that
$$\Lambda^k(A_{j}\#\Omega^{-1})^p=\left(\Lambda^k(A_{j}\#\Omega^{-1})\right)^p=
(\Lambda^kA_{j}\#\Lambda^k\Omega^{-1})^p\geq
(\Lambda^kA_{j}^{1/2})^{p}=\Lambda^{k} A_{j}^{p/2}.$$ Together with
(\ref{eq35}) this gives
\begin{eqnarray*}
G( w ;\Lambda^kA_{1}^{p/2},\dots,\Lambda^kA_{m}^{p/2})&\leq&
\Lambda^kG( w ;(A_{1}\#\Omega^{-1})^p,\dots,(A_{m}\#\Omega^{-1})^p)\leq
I
\end{eqnarray*}
for all $0<p<1.$ In other words,
$$\Lambda^kG( w ;A_{1}^{p/2},\dots,A_{m}^{p/2})\leq
I, \qquad 0<p<1.$$ Raise both sides to their $2/p$ power to get
$$\Lambda^kG( w ;A_{1}^{p/2},\dots,A_{m}^{p/2})^{2/p}\leq
I, \qquad 0<p<1.$$ Now let $p\to 0$ and use Proposition 3. This
shows that $\Lambda^kL\leq I.$

We have shown that for $1\leq k\leq n,$ the condition
$\Lambda^k\Omega\leq I$ implies that $\Lambda^kL\leq I.$ From this
we conclude that
$$\lambda_1^{\downarrow}(\Lambda^kL)\leq
\lambda_1^{\downarrow}(\Lambda^k\Omega),$$ that is,
$$\prod_{j=1}^k\lambda_j^{\downarrow}(L)\leq
\prod_{j=1}^k\lambda_j^{\downarrow}(\Omega).$$ This proves the
second part of Theorem 1.

The easiest part of the proof of Theorem 2 is that of the case
$p=\infty.$ Suppose $\Omega\leq I.$ Then $I\leq \Omega^{-1}$, and by
the monotonicity property of the geometric mean
$$\sum_{j=1}^mw_{j}(A_{j}\#I)\leq
\sum_{j=1}^mw_{j}(A_{j}\#\Omega^{-1}).$$ The left hand side equals
$\sum_{j=1}^{m}w_jA_{j}^{1/2}$, and by (\ref{eq34}) the right hand
side equals $I.$ So the inequality says  $Q_{1/2}^{1/2}\leq I,$
and hence $Q_{1/2}\leq I.$ We have shown that $\Omega\leq I$ implies
that $Q_{1/2}\leq I.$ Hence $$||Q_{1/2}||_{\infty}\leq
||\Omega||_{\infty}.$$

The proof we offer for $p=1$ is more intricate. Given
$A_{1},\dots,A_{m}$ define for each $A\in {\Bbb P}$
\begin{equation}
K(A)=A^{-1/2}\left(\sum_{j=1}^{m}w_{j}(A^{1/2}A_{j}A^{1/2})^{1/2}\right)^2A^{-1/2}.\label{eq36}
\end{equation} It has been shown in \cite{abc} (see also Theorem 11
in \cite{BJL}) that for every $S_0\in {\Bbb P}$ the sequence
$S_{k+1}=K(S_{k})$ converge to $\Omega$ and ${\mathrm{tr}}S_{k}\leq
{\mathrm{tr}}S_{k+1}\leq {\mathrm{tr}}\Omega$ for all $k\geq 1.$ The
special choice $S_{0}=I$ gives $S_{1}=K(I)=Q_{1/2}.$ Hence
${\mathrm{tr}}Q_{1/2}\leq {\mathrm{tr}}\Omega,$ and therefore the
inequality (\ref{eq24}) is valid for $p=1.$

In the special case $m=2$ we have a simple proof. In this case
$ w =(1-t,t)$ for some $0<t<1,$
\begin{eqnarray}
Q_{1/2}( w ;A,B)&=&\left((1-t)A^{1/2}+tB^{1/2}\right)^2\nonumber\\
&=&(1-t)^2A+t^2B+(1-t)t(A^{1/2}B^{1/2}+B^{1/2}A^{1/2}),\label{eq37}
\end{eqnarray}
and $\Omega( w ;A,B)=A\lozenge_t B$ as given in (\ref{eq4}). So
the inequality $||Q_{1/2}||_{1}\leq ||\Omega||_1$ will be
established if we can show that
\begin{equation}
{\mathrm{tr}} A^{1/2}B^{1/2}\leq {\mathrm{tr}}
(AB)^{1/2}.\label{eq38}
\end{equation}
Since ${\mathrm{tr}} (AB)^{1/2}={\mathrm{tr}} A^{-1/2}(AB)^{1/2}A^{1/2}={\mathrm{tr}}
(A^{1/2}BA^{1/2})^{1/2},$ this inequality can also be stated as
$${\mathrm{tr}}\,  A^{1/4}B^{1/2}A^{1/4}\leq {\mathrm{tr}}\,  (A^{1/2}BA^{1/2})^{1/2}.$$
This follows from Theorem IX.2.10 in \cite{rbh}.

To prove the last statement of Theorem 2 we have to show that
\begin{equation}
||Q_{1/2}( w ;A,B)||_{2}^2\leq
||\Omega( w ;A,B)||_{2}^2.\label{eq39}
\end{equation}
Using (\ref{eq37}) and the cyclicity of trace, the left hand side of
(\ref{eq39}) is seen to be equal to
\begin{eqnarray*}
&{}&{\mathrm{tr}}\, [(1-t)^4A^2+4(1-t)^2t^2AB+t^4B^2\\
&{}&\,\,+4(1-t)^3tA^{3/2}B^{1/2}+4(1-t)t^3A^{1/2}B^{3/2}+2(1-t)^2t^2(A^{1/2}B^{1/2})^2].
\end{eqnarray*}
To expand the right hand side we use the expression (\ref{eq4}) for
$\Omega,$ cyclicity of trace, and the observation
\begin{eqnarray*}
{\mathrm{tr}}\,  A(AB)^{1/2}&=&{\mathrm{tr}}
AA^{-1/2}(AB)^{1/2}A^{1/2}\\
&=&{\mathrm{tr}}\,  A(A^{1/2}BA^{1/2})^{1/2}\\
&=& {\mathrm{tr}}\,  (A^{1/2}BA^{1/2})^{1/2}A={\mathrm{tr}}
(BA)^{1/2}A.
\end{eqnarray*}
A little calculation shows that the right hand side of (\ref{eq39})
is equal to
\begin{eqnarray*}
&{}&{\mathrm{tr}}\, [(1-t)^4A^2+4(1-t)^2t^2AB+t^4B^2\\
&{}&+4(1-t)^3tA(AB)^{1/2}+4(1-t)t^3B(BA)^{1/2}\\
&{}&+2(1-t)^2t^2(AB)^{1/2}(BA)^{1/2}].
\end{eqnarray*}
Thus to prove (\ref{eq31}) we need two inequalities that are of
independent interest and are stated in the following proposition.

\begin{prop}
For all positive definite matrices $A$ and $B$ we have
\begin{eqnarray}
&(i)& {\mathrm{tr}}\,  (A^{1/2}B^{1/2})^2\leq {\mathrm{tr}}\,
(AB)^{1/2}(BA)^{1/2},\label{eq40}\\
&(ii)& {\mathrm{tr}} \, A^{3/2}B^{1/2}\leq {\mathrm{tr}}\,
A(AB)^{1/2}.\label{eq41}
\end{eqnarray}
\end{prop}
\begin{proof}
We will show that
\begin{equation}
\tr\, (A^{1/2}B^{1/2})^{2}\leq \tr\, AB \leq \tr\,
(AB)^{1/2}(BA)^{1/2}.\label{eq42}
\end{equation}
The first inequality in (\ref{eq42}) is the famous Lieb-Thirring
inequality (see IX.62 in \cite{rbh}). To prove the second
note that
\begin{eqnarray*}
\tr\, (AB)^{1/2}(BA)^{1/2}&=&||(AB)^{1/2}||_2^2\\
&\geq&\sum_{j=1}^{n}\left[\lambda_{j}(AB)^{1/2}\right]^{2}=\sum_{j=1}^{n}\lambda_{j}(AB)\\
&=&\tr\, AB.
\end{eqnarray*}
This proves the inequality (\ref{eq42}).

The inequality (\ref{eq41}) can be also be stated as
$$\tr\, A^{3/4}B^{1/2}A^{3/4}\leq \tr\,
A^{1/2}(A^{1/2}BA^{1/2})^{1/2}A^{1/2}.$$ We will prove a much
stronger log majorisation:
\begin{equation}
\lambda\left(A^{3/4}B^{1/2}A^{3/4}\right)\underset{\log}{\prec}\lambda\left(A^{1/2}(A^{1/2}BA^{1/2})^{1/2}A^{1/2}\right).\label{eq43}
\end{equation}
We first prove
\begin{equation}
\lambda_1^{\downarrow}\left(A^{3/4}B^{1/2}A^{3/4}\right)\leq
\lambda_1^{\downarrow}\left(A^{1/2}(A^{1/2}BA^{1/2})^{1/2}A^{1/2}\right).\label{eq44}
\end{equation}
As explained earlier, for this it suffices to prove the implication
$$A^{1/2}(A^{1/2}BA^{1/2})^{1/2}A^{1/2}\leq I\Longrightarrow
A^{3/4}B^{1/2}A^{3/4}\leq I.$$ This is equivalent to the statement
\begin{equation}
(A^{1/2}BA^{1/2})^{1/2}\leq A^{-1}\Longrightarrow B^{1/2}\leq
A^{-3/2}.\label{eq45}
\end{equation}
Here we invoke Furuta's inequality. This tells us that
$$X\geq Y\geq 0\Longrightarrow (X^{p+2r})^{1/p}\geq
(X^rY^pX^r)^{1/p}$$ for all $p\geq 1, r\geq 0.$ (See Corollary 4.4.2
in \cite{rbh1}). Choosing
$X=A^{-1},Y=(A^{1/2}BA^{1/2})^{1/2},p=2,r=1/2$ in Furuta's
inequality, we obtain the assertion \eqref{eq45}. This gives us the
inequality (\ref{eq44}).

Applying this to $k\textrm{th}$ antisymmetric tensor powers of the matrices
involved we obtain
$$\prod_{j=1}^{k}\lambda_j^{\downarrow}(A^{3/4}B^{1/2}A^{3/4})\leq
\prod_{j=1}^{k}\lambda_j^{\downarrow}(A^{1/2}(A^{1/2}BA^{1/2})^{1/2}A^{1/2}),$$
for $1\leq k\leq n.$ For $k=n$ there is equality here as both sides
are equal to ${\mathrm{det}}(A^{3/2}B^{1/2}).$ This proves
(\ref{eq43}), and as a corollary (\ref{eq41}).
\end{proof}

All the assertions of Theorem 2 have been established.

\section{Remarks}

The reader would have noticed that our Conjecture $1$ asserts only
weak majorisation and not weak log majorisation in (\ref{eq25}). In
\cite{au} Audenaert has shown the following determinant inequality
in the case $m=2$ and equal weights $(w_{1},w_{2})=(1/2,1/2),$
\begin{equation}
\det\,\Omega(A,B)\leq\det\,Q_{1/2}(A,B).\label{eq46}
\end{equation}
This goes in the direction opposite to (\ref{eq25}), and opposite
to what weak log majorisation would have implied.

It is reasonable to conjecture that the inequality (\ref{eq46})
remains true in the several variable case as well.

Theorem 2 supplements some results proved in \cite{bly}. In their
analysis of the Cartan mean, Lim and Palfia \cite{lp1} introduced
another version of the power mean. They showed that for $0<t<1,$
the equation
\begin{equation}
X=\sum_{j=1}^{m}w_{j}(X\#_tA_{j}),\label{eq47}
\end{equation}
has a unique positive definite solution. Call this $P_{t}( w
;A_{1},\dots,A_{m}).$ When $A_{1},\dots,A_{m}$ commute,
$P_{t}=\left(\sum_{j=1}^mw_{j}A_j^{t}\right)^{1/t}=Q_{t}.$ Lim and
Palfia showed that
\begin{equation}
{\underset{p\to
0}{\lim}}P_{t}( w ;A_{1},\dots,A_{m})=G( w ;A_{1},\dots,A_{m}).\label{eq48}
\end{equation}
Compare this with (\ref{eq21}).

It was conjectured in \cite{bly} that
\begin{equation}|||P_{t}( w ;A_1,\dots,A_{m})|||\leq
|||Q_{t}( w ;A_1,\dots,A_{m})|||,\label{eq49}
\end{equation}
for all $0<t<1.$ The case $||\cdot||_{\infty}$ of this had been proved in
\cite{ly} and the case $||\cdot||_1$ was proved in \cite{bly}. The
case $0<p<\infty$ has been proved recently in \cite{ddf}. Explicit
formulas for $P_t$ are known only when $m=2.$ In the special case of
equal weights $(w_{1},w_{2})=(1/2,1/2)$ we have
\begin{equation}
P_{1/2}(A,B)=\frac{1}{4}(A+B+2(A\#B)).\label{eq50}
\end{equation}
So in this special case the conjecture (\ref{eq49}) says
\begin{equation}
|||A+B+2(A\#B)|||\leq
|||A+B+A^{1/2}B^{1/2}+B^{1/2}A^{1/2}|||.\label{eq51}
\end{equation}
In addition to the $||\cdot||_p$ norms, $p=1,\infty,$ this was shown
to be true for $p=2$ in \cite{bly}. Compare the statements
(\ref{eq26}) and (\ref{eq51}).

In the course of the proofs in \cite{bly} certain log majorisations
complementary to (\ref{eq43}) have been proved. Together these say
\begin{equation}
\lambda\left(A^2(A^{-1}B)^{1/2}\right)\underset{\log}{\prec}\lambda\left(A^{3/2}B^{1/2}\right)
\underset{\log}{\prec}\lambda\left(A(AB)^{1/2}\right).\label{eq52}
\end{equation}
This implies the trace inequalities
\begin{equation}\tr\, A^{2}(A^{-1}B)^{1/2}\leq \tr\, A^{3/2}B^{1/2}\leq
\tr\,A(AB)^{1/2}.\label{eq53} \end{equation} These inequalities
illustrate the effect of rearranging factors in noncommutative
products. There are several inequalities of this kind  that are
known, the most famous being the Golden-Thompson and the
Lieb-Thirring inequalities.

In view of the comparison between $L$ and $\Omega$ given in
(\ref{eq19}) and between $Q_{1/2}$ and $\Omega$ conjectured in
(\ref{eq25}), we may ask how $L$ and $Q_{1/2}$ compare with each
other, It has been shown in \cite{bg} that for each $1\leq j\leq n,$
$\lambda_j^{\downarrow}(Q_{p}( w ;A_{1},\dots,A_{m}))$ is a
monotonically increasing function of $p$ on $(-\infty,\infty).$ In
particular, using (\ref{eq21}) we see from this
\begin{equation}
\lambda_j^{\downarrow}(L( w ;A_{1},\dots,A_{m}))\leq
\lambda_j^{\downarrow}(Q_{1/2}( w ;A_{1},\dots,A_{m})), \qquad
1\leq j\leq n.\label{eq54}
\end{equation}

\section{Acknowledgements.} The first author is a J. C. Bose National Fellow.
The work of Y.~Lim was supported by the National Research Foundation
of Korea (NRF) grant funded by the Korea government(MEST)
No.2015R1A3A2031159 and 2016R1A5A1008055.



\end{document}